\documentclass[a4paper, 11pt, reqno]{amsart}
\usepackage{amsmath}
\usepackage{amsfonts}
\usepackage{amssymb}
\usepackage{graphicx}
\usepackage{amsthm}
\usepackage[hyperpageref]{backref}
\usepackage{hyperref}

\usepackage[latin1]{inputenc}

\usepackage{color}
\usepackage{url}
\usepackage{bigints}

\newtheorem{theorem}{Theorem}

\newtheorem{definition}[theorem]{Definition}

\newtheorem{lemma}[theorem]{Lemma}

\makeatletter
\newcommand{\addresseshere}{%
  \enddoc@text\let\enddoc@text\relax
}
\makeatother

\newcommand{\rr}{\mathbb{R}}

\newcommand{\Hess}{\nabla^2}

\newcommand{\dive}{\mathrm{div}}
\newcommand{\Tr}{\mathrm{Tr}}
\newcommand{\dist}{\mathrm{dist}}

\usepackage[margin=1in]{geometry}

\begin{document}
\title{Serrin's type overdetermined problems in convex cones}
\author{Giulio Ciraolo and Alberto Roncoroni}
\thanks{}
\address{Giulio Ciraolo, Dipartimento di Matematica e Informatica, Universit\`a degli Studi di Palermo, Via Archirafi 34, 90123, Palermo, Italy}
\email{giulio.ciraolo@unipa.it}
\address{Alberto Roncoroni, Dipartimento di Matematica F. Casorati, Universit\`a degli Studi di Pavia, Via Ferrata 5, 27100 Pavia, Italy}
\email{alberto.roncoroni01@universitadipavia.it}

\date{\today}

\begin{abstract}
We consider overdetermined problems of Serrin's type in convex cones for (possibly) degenerate operators in the Euclidean space as well as for a suitable generalization to space forms. We prove rigidity results by showing that the existence of a solution implies that the domain is a spherical sector. 
\end{abstract}

\maketitle

\section{Introduction}
Given a bounded domain $E \subset  \mathbb{R}^N$, $N\geq 2$, the classical Serrin's overdetermined problem \cite{Serrin} asserts that there exists a solution to 
\begin{equation} \label{pb_serrin}
\begin{cases}
\Delta u = -1 & \textmd{ in } E \,, \\
u=0 & \textmd{ on } \partial E \,, \\
\partial_{\nu}u = -c & \textmd{ on } \partial E \,,
\end{cases}
\end{equation} 
for some constant $c>0$, if and only if $E=B_R(x_0)$ is a ball of radius $R$ centered at some point $x_0$. Moreover, the solution $u$ is radial and it is given by 
\begin{equation} \label{u_radial}
u(x) = \frac{R^2 - |x-x_0|^2}{2N}\,,
\end{equation}
with $R=Nc$. Here, $\nu$ denotes the outward normal to $\partial \Omega$.

The starting observation of this manuscript is the following. 
Let $\Sigma$ be an open cone in $\mathbb{R}^N$ with vertex at the origin $O$, i.e.
\begin{equation*}
\Sigma=\lbrace tx \, : \, x\in\omega, \, t\in(0,+\infty)\rbrace
\end{equation*}
for some open domain $\omega\subset \mathbb{S}^{N-1}$. We notice that if $x_0$ is chosen appropriately then $u$ given by \eqref{u_radial} is still the solution to 
\begin{equation*} 
\begin{cases}
\Delta u = -1 & \textmd{ in } B_R(x_0) \cap \Sigma \,, \\
u=0 \text{ and } \partial_\nu u = -c & \textmd{ on } \partial B_R(x_0) \setminus \overline{\Sigma}\,, \\
\partial_{\nu}u = 0 & \textmd{ on } B_R(x_0) \cap \partial \Sigma \,.
\end{cases}
\end{equation*} 
More precisely, $x_0$ may coincide with $O$ or it may be just a point of $\partial \Sigma \setminus \{O\}$ and, in this case, $B_R(x_0) \cap \Sigma$ is half a sphere lying over a flat portion of $\partial \Sigma$. Hence, it is natural to look for a characterization of symmetry in this direction, as done in \cite{Pacella-Tralli} (see below for a more detailed description).

In order to properly describe the results, we introduce some notation. Given an open cone $\Sigma$ such that $\partial\Sigma\setminus\lbrace O\rbrace$ is smooth, we consider a bounded domain $\Omega\subset\Sigma$ and denote by $\Gamma_0$ its relative boundary, i.e. 
$$
\Gamma_0 = \partial \Omega \cap \Sigma \,,
$$
and we set
$$
\Gamma_1=\partial\Omega\setminus\bar{\Gamma}_0 \,.
$$ 
We assume that $\mathcal{H}_{N-1}(\Gamma_1)>0$, $\mathcal{H}_{N-1}(\Gamma_0)>0$ and that $\Gamma_0$ is a smooth $(N-1)$-dimensional manifold, while $\partial\Gamma_0=\partial\Gamma_1 \subset\partial\Omega\setminus\lbrace O\rbrace$ is a smooth $(N-2)$-dimensional manifold. Following \cite{Pacella-Tralli}, such a domain $\Omega$ is called a \emph{sector-like domain}. In the following, we shall write $\nu=\nu_x$ to denote the exterior unit normal to $\partial\Omega$ wherever is defined (that is for $x\in\Gamma_0\cup\Gamma_1\setminus\lbrace O\rbrace$). 


Under the assumption that $\Sigma$ is a convex cone, in \cite{Pacella-Tralli} it is proved that if $\Omega$ is a sector-like domain and there exists a classical solution $u \in C^2(\Omega) \cap C^1(\Gamma_0 \cup \Gamma_1 \setminus \{O\})$ to 
\begin{equation} \label{pb_serrin_cone_PT}
\begin{cases}
\Delta u = -1 & \textmd{ in } \Omega \,, \\
u=0 \textmd{ and } \partial_\nu u = -c & \textmd{ on } \Gamma_0\,, \\
\partial_{\nu}u = 0 & \textmd{ on }  \Gamma_1 \setminus \{O\} \,,
\end{cases}
\end{equation} 
and such that 
$$
u \in W^{1,\infty}(\Omega) \cap W^{2,2}(\Omega) \,,
$$
then 
$$
\Omega = B_R(x_0) \cap \Sigma
$$ 
for some $x_0 \in \mathbb{R}^N$ and $u$ is given by \eqref{u_radial}. Differently from the original paper of Serrin \cite{Serrin}, the method of moving planes is not helpful (at least when applied in a standard way) and the rigidity result in \cite{Pacella-Tralli} is proved by using two alternative approaches. One is based on integral identities and it is inspired from \cite{BNST}, the other one uses a $P$-function approach as in \cite{Weinberger}. 


%
%

In this paper, we generalize the rigidity result for Serrin's problem in \cite{Pacella-Tralli} in two directions. The former is by considering more general operators than the Laplacian in the Euclidean space, where the operators may be of degenerate type. Here, the generalization is not trivial due to the lack of regularity of the solution (the operator may be degenerate) as well as to other technical details which are not present in the linear case.  

The latter is by considering an analogous problem in space forms, i.e. the hyperbolic space and the (hemi)sphere. The operator that we consider is linear and it is interesting since it has been shown that it is a helpful generalization of the torsion problem to space forms (\cite{CV1}, \cite{QuiXia}, \cite{QX}).

\medskip

\noindent {\bf More general operators in the Euclidean space.} 
Let $\Omega$ be a sector like domain in $\mathbb{R}^N$ and let $f:[0,+\infty) \to [0,+\infty)$ be such that 
\begin{equation} \label{f_HP}
\begin{aligned}
f \in C([0,\infty))\cap C^3(0,\infty) & \textmd{ with } f(0)=f'(0)=0, \ f''(s)>0 \text{ for } s>0 \\
& \text{ and } \lim_{s\rightarrow +\infty }\dfrac{f(s)}{s}=+ \infty \,.
\end{aligned}
\end{equation}
We consider the following mixed boundary value problem
\begin{equation}\label{pb cono}
\begin{cases}
L_f u=-1 &\mbox{in } \Omega, \\ u=0 &\mbox{on } \Gamma_0 \\ \partial_{\nu}u=0 \, &\mbox{on } \Gamma_1\setminus\lbrace O\rbrace,
\end{cases}
\end{equation}
where the operator $L_f$ is given by 
\begin{equation}\label{operator}
L_f u=\dive\left(f'(|\nabla u|)\dfrac{\nabla u}{|\nabla u|}\right),
\end{equation}
and the equation $L_f u = -1$ is understood in the sense of distributions
\begin{equation*}
\int_\Omega\dfrac{f'(|\nabla u|)}{|\nabla u|}\nabla u\cdot\nabla\varphi\, dx=\int_{\Omega}\varphi\, dx
\end{equation*}
for any 
\begin{equation*}
\varphi\in T(\Omega):=\lbrace \varphi\in C^1(\Omega) \, : \, \varphi\equiv 0 \, \textit{ on } \Gamma_0\rbrace.
\end{equation*}
Notice that the operator $L_f$ may be of degenerate type.

We notice that the solution to $L_f u =-1$ in $B_R(x_0)$ (a ball of radius $R$ centered at $x_0$) such that $u= 0$ on $\partial B_R(x_0)$ is radial and it is given by
\begin{equation} \label{u_radial_Lf}
u(x)=\int_{|x-x_0|}^{R}g'\left(\dfrac{s}{N}\right)\, ds \,,
\end{equation}
where $g$ denotes the Fenchel conjugate of $f$ (see for instance \cite{Crasta} or \cite{FGK}), i.e.
$$
g=\sup\lbrace st-f(s) \, : \, s\geq 0\rbrace
$$
(hence for us $g'$ is the inverse function of $f'$).
Our first main result is the following. 

\begin{theorem}\label{teo 1 cono}
Let $f$ satisfy \eqref{f_HP}. Let $\Sigma$ be a convex cone such that $\Sigma\setminus\lbrace O\rbrace$ is smooth and let $\Omega$ be a sector-like domain in $\Sigma$.  If there exists a solution $u\in C^1(\Omega\cup\Gamma_0\cup\Gamma_1\setminus\lbrace O\rbrace)\cap W^{1,\infty}(\Omega)$ to \eqref{pb cono} such that
\begin{equation}\label{overdetermined cond}
\partial_{\nu}u=-c \, \textit{ on } \, \Gamma_0
\end{equation}
for some constant $c$, and satisfying
\begin{equation}\label{key property}
\dfrac{f'(|\nabla u|)}{|\nabla u|}\nabla u \in W^{1,2}(\Omega,\mathbb{R}^N) \,,
\end{equation}
then there exists $x_0 \in \mathbb{R}^N$ such that  $\Omega=\Sigma\cap B_{R}(x_0)$ with $c=g'(|\Omega|/|\Gamma_0|)$, $R=N|\Omega|/|\Gamma_0|$. Moreover $u$ is given by \eqref{u_radial_Lf}, where $x_0$ is the origin or, if $\partial \Sigma$ contains flat regions, it is a point on $\partial \Sigma$.
\end{theorem}

When $L_f = \Delta$ (i.e. $f(t)=t^2/2$),  Theorem \ref{teo 1 cono} is essentially Theorem 1.1 in \cite{Pacella-Tralli}.  Condition  \eqref{key property} holds (at least locally in $\Omega$) for  uniformly elliptic operators, such as the mean curvature operator ($f(t) = \sqrt{1+t^2}$), and also for degenerate operators such as the $p-$Laplace operator ($f(t)=t^p/p$), see \cite{Mingione} and \cite{CiMa}. We stress that the validity of \eqref{key property} up to the boundary is more subtle, since it depends strongly on how $\Gamma_0$ and $\Gamma_1$ intersect. 

We observe that the overdetermined problem \eqref{pb cono} with the condition \eqref{overdetermined cond} can be seen as a partially overdetermined problem (see for instance \cite{FV1} and \cite{FV2}), since we impose both Dirichlet and Neumann conditions only on a part of the boundary, namely $\Gamma_0$, while a sole homogeneous Neumann boundary condition is assigned on $\Gamma_1$ (where, however, there is the strong assumption that it is contained in the boundary of a cone).

We notice that the proof of Theorem \ref{teo 1 cono} still works when $\Gamma_1 = \emptyset$ (hence $\partial\Omega=\Gamma_0$). In this case we obtain the celebrated result of Serrin \cite{Serrin} for the operator $L_f$ (see also \cite{BC}, \cite{BNST}, \cite{CMS}, \cite{FGK}, \cite{FarinaKawhol}, \cite{Garofalo-Lewis}, \cite{Roncoroni}, \cite{Weinberger}).
Moreover, the proof is also suitable to be adapted to the anisotropic counterpart of the overdetermined problem \eqref{pb cono} and \eqref{overdetermined cond} by following the approach used in this manuscript and in \cite{BC} (see also \cite{CS} and \cite{WX}). 
We also mention that rigidity theorems in cones are related to the study of relative isoperimetric and Sobolev inequalities in cones, and we refer to \cite{Pacella-Tralli} for a more detailed discussion (see also \cite{Baer-Figalli,Cabre-RosOton-Serra,Figalli-Indrei,Grossi-Pacella,Lions-Pacella-Tricarico,Lions-Pacella}).

\medskip

\noindent {\bf Serrin's problem in cones in space forms.} A space form is a complete simply-connected Riemannian manifold $(M,g)$ with constant sectional curvature $K$. Up to homotheties we may assume $K=0$,$1$,$-1$: the case $K=0$ corresponds to the Euclidean space $\mathbb{R}^N$, $K=-1$ is the hyperbolic space $\mathbb{H}^N$ and $K=1$ is the unitary sphere with the round metric $\mathbb{S}^N$. More precisely, in the case $K=1$ we consider the hemisphere $\mathbb{S}^N_+$. These three models can be described as warped product spaces $M=I\times \mathbb S^{N-1}$ equipped with the rotationally symmetric metric 
$$
g=dr^2+h(r)^2\,g_{\mathbb S^{N-1}},
$$  
where $g_{\mathbb S^{N-1}}$ is the round metric on the $(N-1)$-dimensional sphere $\mathbb S^{N-1}$ and 
\begin{enumerate}
	\item[-] $h(r)=r$ in the Euclidean case ($K=0$), with $I=[0,\infty)$;

	\item[-] $h(r)=\sinh(r)$ in the hyperbolic case ($K=-1$), with $I=[0,\infty)$;
	
	\item[-] $h(r)=\sin(r)$ in the spherical case ($K=1$), with $I=[0,\pi/2)$ for $\mathbb{S}^N_+$. 
\end{enumerate}

By using the warping structure of the manifold, we denote by $O$ the pole of the model and it is natural to define a \emph{cone $\Sigma$ with vertex at $\{O\}$} as the set
\begin{equation*}
\Sigma=\lbrace tx \, : \, x\in\omega, \, t\in I \rbrace
\end{equation*}
for some open domain $\omega\subset \mathbb{S}^{N-1}$. Moreover, we say that $\Sigma$ is a \emph{convex cone} if the second fundamental form $\mathrm {II}$ is nonnegative defined at every $p \in \partial \Sigma$. 

Serrin's overdetermined problem for semilinear equations $\Delta u + f(u) =0$ in space forms has been studied in \cite{Kumaresan_Prajapat} and \cite{Mol} by using the method of moving planes. 
If one considers the corresponding problem for sector-like domains in space forms, the method of moving planes can not be used and one has to look for alternative approaches. As already mentioned, in the Euclidean space these approaches typically use integral identities and $P$-functions (see \cite{BNST,Weinberger}) and have the common feature that at a crucial step of the proof they use the fact that the radial solution attains the equality sign in a Cauchy-Schwartz inequality, which implies that the Hessian matrix $\Hess u $ is proportional to the identity. Since the equivalent crucial step in space forms is to prove that the Hessian matrix of the solution is proportional to the metric, then the equation $\Delta u =-1$ is no more suitable (one can easily verify that in the radial case the Hessian matrix of the solution is not proportional to the metric) and a suitable equation to be considered is 
\begin{equation} \label{eq_space_forms}
\Delta u+NKu=-1 
\end{equation}
as done in \cite{CV1} and \cite{QuiXia}, \cite{QX}. It is clear that for $K=0$, i.e. in the Euclidean case, the equation reduces to $\Delta u = -1$.
For this reason, we believe that, in this setting, \eqref{eq_space_forms} is the natural generalization of the Euclidean $\Delta u=-1$ to space forms.

A Serrin's type rigidity result for \eqref{eq_space_forms} can be proved following Weinberger's approach by using a suitable $P$-function associated to \eqref{eq_space_forms} (see \cite{CV1} and \cite{QX}). This approach is helpful for proving the following Serrin's type rigidity result for convex cones in space forms, which is the second main result of this paper. 

\begin{theorem}\label{teo 2 cono}
Let $(M,g)$ be the Euclidean space, hyperbolic space or the hemisphere.	Let $\Sigma \subset M$ be a convex cone such that $\Sigma\setminus\lbrace O\rbrace$ is smooth and let $\Omega$ be a sector-like domain in $\Sigma$.  Assume that there exists a solution $u\in C^1(\Omega\cup\Gamma_0\cup\Gamma_1\setminus\lbrace O\rbrace)\cap W^{1,\infty}(\Omega) \cap W^{2,2}(\Omega)$ to
\begin{equation}\label{pb cono sf}
	\begin{cases}
	\Delta u+NKu=-1 &\mbox{in } \Omega, \\ u=0 &\mbox{on } \Gamma_0 \\ \partial_{\nu}u=0 \, &\mbox{on } \Gamma_1\setminus\lbrace O\rbrace,
	\end{cases}
	\end{equation}
	 such that
	\begin{equation}\label{overdetermined cond spaceforms}
	\partial_{\nu}u=-c \, \textit{ on } \, \Gamma_0
	\end{equation}	
	for some constant $c$.	Then $\Omega=\Sigma\cap B_{R}(x_0)$ where $B_R(x_0)$ is a geodesic ball of radius $R$ centered at $x_0$ and $u$ is given by
$$
u(x)=\frac{H(R)-H(d(x,x_0))}{n \dot h(R)} \,,
$$
with 
$$
H(r)=\int_0^r h(s) ds 
$$ 
and where $d(x,x_0)$ denotes the distance between $x$ and $x_0$.
\end{theorem}

\medskip

\medskip

\noindent {\bf Organization of the paper.} The paper is organized as follows: in Section \ref{section_preliminary} we introduce some notation, we recall some basic facts about elementary symmetric function of a matrix and prove some preliminary result needed to prove Theorem \ref{teo 1 cono}. Theorems \ref{teo 1 cono} and \ref{teo 2 cono} are proved in Sections  \ref{section_proofthm1} and \ref{section_spaceforms}, respectively.

\section{Preliminary results for Theorem \ref{teo 1 cono}} \label{section_preliminary}
In this section we collect some preliminary results which are needed in the proof of Theorem \ref{teo 1 cono}. Let $f$ satisfy \eqref{f_HP} and consider problem \eqref{pb cono}
\begin{equation*}
\begin{cases}
L_f u=-1 &\mbox{in } \Omega, \\ u=0 &\mbox{on } \Gamma_0 \\ \partial_{\nu}u=0 \, &\mbox{on } \Gamma_1\setminus\lbrace O\rbrace,
\end{cases}
\end{equation*}
where the operator $L_f$ is given by 
\begin{equation*} 
L_f u=\dive\left(f'(|\nabla u|)\dfrac{\nabla u}{|\nabla u|}\right) \,.
\end{equation*}
 
\begin{definition}
$u\in C^1(\Omega\cup\Gamma_0\cup\Gamma_1\setminus\lbrace O\rbrace)$ is a 
solution to Problem \eqref{pb cono} if
\begin{equation}\label{weak sol cono}
\int_\Omega\dfrac{f'(|\nabla u|)}{|\nabla u|}\nabla u\cdot\nabla\varphi\, dx=\int_{\Omega}\varphi\, dx
\end{equation}
for any 
\begin{equation}
\varphi\in T(\Omega):=\lbrace \varphi\in C^1(\Omega) \, : \, \varphi\equiv 0 \, \textit{ on } \Gamma_0\rbrace.
\end{equation}
\end{definition}
We observe some facts that will be useful in the following. Since the outward normal $\nu$ to $\Gamma_0$ is given by 
	\begin{equation}\label{normale}
			\nu=-\frac{\nabla u}{|\nabla u|}|_{\Gamma_0} \,,
	\end{equation}
	then \eqref{overdetermined cond} implies that
	\begin{equation}\label{gradiente cost}
	|\nabla u|=c \quad \text{ on } \quad \Gamma_0.
	\end{equation}
Moreover we observe that the constant $c$ in the statement is given by
\begin{equation}\label{valore_c}
c=g'\left(\frac{|\Omega|}{|\Gamma_0|}\right),
\end{equation}
as it follows by integrating the equation $L_fu=-1$, by using the divergence theorem, formula \eqref{gradiente cost} and the fact that $\partial_{\nu}u=0$ on $\Gamma_1\setminus\{ O\}$.
We also notice that 
\begin{equation}
x\cdot\nu=0 \quad \text{ on } \quad \Gamma_1.
\end{equation}

It will be useful to write the operator $L_f$ as the trace of a matrix. Let $V:\rr^N\rightarrow\rr$ be given by
\begin{equation}\label{definizione di V}
V(\xi)=f(|\xi|) \quad \text{for } \xi \in \rr^N,
\end{equation}
and notice that 
\begin{equation}\label{derivatediV}
V_{\xi_i}(\xi)=f'(|\xi|)\dfrac{\xi_i}{|\xi|} \quad \text{ and } \quad V_{\xi_i\xi_j}(\xi)=f''(|\xi|)\dfrac{\xi_i\xi_j}{|\xi|^2}-f'(|\xi|)\dfrac{\xi_i\xi_j}{|\xi|^3}+f'(|\xi|)\dfrac{\delta_{ij}}{|\xi|}.
\end{equation}
Hence, by setting 
$$
W=(w_{ij})_{i,j=1,\dots,N}
$$ 
where 
\begin{equation} \label{Wdef}
w_{ij}(x)=\partial_{j} V_{\xi_i}(\nabla u(x)) \,,
\end{equation} 
we have 
\begin{equation} \label{Lfu_trW}
L_f(u) = \Tr(W).
\end{equation}
Notice that at regular points, where $\nabla u \neq 0$, it holds that 
\begin{equation} \label{Wdef_II}
W=\nabla^2_{\xi}V(\nabla u)\nabla^2 u \,.
\end{equation}
Our approach to prove Theorem \ref{teo 1 cono} is to write several integral identities and just one pointwise inequality, involving the matrix $W$. Writing the operator $L_f$ as trace of $W$ has the advantage that we can use the generalization of the so-called Newton's inequalities, as explained in the following subsection.

\subsection{Elementary symmetric functions of a matrix}
Given a matrix $A=(a_{ij})\in\rr^{N\times N}$, for any $k=1,\dots,N$ we denote by $S_k(A)$ the sum of all the principal minors of $A$ of order $k$. In particular, $S_1(A)=\Tr(A)$ is the trace of $A$, and $S_n(A)=\mathrm{det}(A)$ is the determinant of $A$. We consider the case $k=2$. By setting
\begin{equation}\label{def_S^2}
S^2_{ij}(A)=-a_{ji}+\delta_{ij}\Tr(A),
\end{equation}
we can write
\begin{equation}\label{defS_2}
S_2(A)=\frac{1}{2}\sum_{i,j}S^2_{ij}(A)a_{ij}=\frac{1}{2}((\Tr(A)^2-\Tr(A^2)) \,.
\end{equation}
The elementary symmetric functions of a symmetric matrix $A$ satisfy the so called Newton's inequalities. In particular, $S_1$ and $S_2$ are related by
\begin{equation}\label{CS per S_2}
S_2(A)\leq \frac{N-1}{2N}(S_1(A))^2 \,.
\end{equation}
When the matrix $A=W$, with  $W$ given by \eqref{Wdef_II}, we have
\begin{equation}\label{defS_ij}
S^2_{ij}(W)=-V_{\xi_j\xi_k}(\nabla u)u_{ki}+\delta_{ij}L_f u \,,
\end{equation}
and $S_{ij}^2(W)$ is divergence free in the following (weak) sense
\begin{equation}\label{derivata di S nulla}
\dfrac{\partial}{\partial x_j}S_{ij}^2(W)=0 \,.
\end{equation}
We will need a generalization of \eqref{CS per S_2} to not necessarily symmetric matrices, which is given by the following lemma. 
\begin{lemma}[\cite{CS}, Lemma 3.2]\label{lemma matrici generiche}
Let $B$ and $C$ be symmetric matrices in $\rr^{N\times N}$, and let $B$ be positive semidefinite. Set $A=BC$. Then the following inequality holds:
\begin{equation}\label{matrices}
S_2(A)\leq\dfrac{N-1}{2N}\Tr(A)^2.
\end{equation}
Moreover, if $\Tr(A)\neq 0$ and equality holds in \eqref{matrices}, then 
\begin{equation*}
A=\dfrac{\Tr(A)}{N}I,
\end{equation*}
and $B$ is, in fact, positive definite.
\end{lemma}

\subsection{Some properties of solutions to \eqref{pb cono}} 
In this subsection we collect some properties of the solutions to \eqref{pb cono}. We assume that the solution is of class $C^1(\Omega) \cap W^{1,\infty}(\Omega)$. From standard regularity elliptic estimates one has that $u$ is of class $C^{2,\alpha}$ where $\nabla u \neq 0$. If one has more information about the degeneracy at zero of $f$ (see \cite{Mingione} and \cite{CiMa}), then one may conclude that $u \in C^{1,\alpha} ( \Omega)$ as well as that 
\begin{equation*} 
\frac{f'(|\nabla u|)}{|\nabla u|} \nabla u \in W^{1,2}_{loc}(\Omega, \mathbb{\rr}^N) \,.
\end{equation*}
The regularity up to the boundary is more difficult to be understood, and it strongly depends on how $\Gamma_0$ and $\Gamma_1$ intersect. This will be one of the major points of the proof of Theorem \ref{teo 1 cono}.

In the following two lemmas we show that $u>0$ in  $\Omega\cup\Gamma_1$ and we prove a Pohozaev-type identity.

\begin{lemma}\label{prop 2}
Let $f$ satisfy \eqref{f_HP} and let $u$ be a solution of \eqref{pb cono}. Then 
\begin{equation}
u>0 \quad \textit{ in } \quad \Omega\cup\Gamma_1.
\end{equation}
\end{lemma}

\begin{proof}
We write $u=u^+-u^-$ and use $\varphi = u^-$ as test function in \eqref{weak sol cono}:
\begin{equation*}
0 \geq -\int_{\Omega \cap \{u<0\}} \dfrac{f'(|\nabla u|)}{|\nabla u|}\, |\nabla u^-|^2\, dx= \int_{\Omega \cap \{u<0\}} u^- \, dx  \geq 0\,,
\end{equation*}
which implies that $u\geq 0$ in $\Omega$. Moreover, if one assumes that  $u(x_0)=0$ at some point $x_0 \in \Omega \cup \Gamma_1$, then $\nabla u(x_0)=0$, which leads to a contradiction by using the comparison principle between the solution $u$ and the radial solution in a suitable ball.
\end{proof}

\begin{lemma}[Pohozaev-type identity]\label{Pohozaev identity}
Let $\Omega$ be a sector-like domain and assume that $f$ satisfies \eqref{f_HP}. Let $u\in W^{1,\infty}(\Omega)$ be a solution to \eqref{pb cono}. Then the following integral identity
\begin{equation}\label{Pohozaev}
\int_\Omega[(N+1) u-N f(|\nabla u|)]\, dx=\int_{\Gamma_0}[f'(|\nabla u|)|\nabla u|-f(|\nabla u|)]x\cdot\nu\, d\sigma
\end{equation}
holds.
\end{lemma}

\begin{proof}
We argue by approximation. For $t\geq 0$ and $\varepsilon\in(0,1)$, let
\begin{equation*}
f_{\varepsilon}(t)=f(\sqrt{\varepsilon^2+t^2})-f(\varepsilon) \,.
\end{equation*}
We define $F(t)=f'(t)t$ and $F_\varepsilon(t)=f_\varepsilon'(t)t$. From a standard argument (see for instance \cite[Lemma 4.2]{CFV}) we have that 
\begin{equation}\label{convergenze}
f_\varepsilon\rightarrow f \quad \text{and} \quad F_\varepsilon\rightarrow F \quad \text{uniformly on compact sets of  $[0,+\infty)$.}
\end{equation}
We recall that $V(\xi)=f(|\xi|)$ (see \eqref{definizione di V}) for  $\xi\in\mathbb{R}^N$, and we define $V^\varepsilon:\mathbb{R}^N\rightarrow\mathbb{R}$ as
\begin{equation*}
V^\varepsilon(\xi):=f_\varepsilon(|\xi|).
\end{equation*}
We approximate $\Omega$ by domains $\Omega_\delta$ obtained by chopping off a $\delta$-tubular neighborhood of $\partial\Gamma_0$ and a $\delta$-neighborhood of $O$.  For $n \in \mathbb{N}$, we consider $u^n_{\delta}\in C^\infty(\Omega_\delta)\cap C^1(\bar \Omega_\delta)$ such that
\begin{equation*}
u^n_{\delta} \rightarrow u \, \text{ in } \, C^1(\overline{\Omega}_\delta),
\end{equation*}
as $n$ goes to infinity (see for instance \cite[Section 2.6]{Burenkov}). $\\ $
Since
\begin{equation*}
\dive\left(x\cdot\nabla u^n_{\delta}\nabla_\xi V^\varepsilon(\nabla u^n_{\delta})\right)=x\cdot\nabla u^n_{\delta}\dive\left(\nabla_\xi V^\varepsilon(\nabla u^n_{\delta})\right)+\nabla(x\cdot\nabla u^n_{\delta})\cdot\nabla_\xi V^\varepsilon(\nabla u^n_{\delta}) 
\end{equation*}
and from 
\begin{equation*}
\begin{aligned}
\nabla(x\cdot\nabla u^n_{\delta})\cdot\nabla_\xi V^\varepsilon(\nabla u^n_{\delta})=&\nabla u^n_{\delta}\cdot\nabla_\xi V^\varepsilon(\nabla u^n_{\delta})+x\nabla^2(u^n_{\delta})\cdot\nabla_\xi V^\varepsilon(\nabla u^n_{\delta})\\
=&\dive\left(u^n_{\delta}\nabla_\xi V^\varepsilon(\nabla u^n_{\delta})\right)-u^n_{\delta}\dive\left(\nabla_\xi V^\varepsilon(\nabla u^n_{\delta})\right) \\
&+\dive(xV^\varepsilon(\nabla u^n_{\delta}))-NV^\varepsilon(\nabla u^n_{\delta}) \,,
\end{aligned}
\end{equation*}
we obtain
\begin{equation}\label{div_smart}
\dive\left(\varphi_n\nabla_\xi V^\varepsilon(\nabla u^n_{\delta})-xV^\varepsilon(\nabla u^n_{\delta})\right)=\varphi_n\dive\left(\nabla_\xi V^\varepsilon(\nabla u^n_{\delta})\right)-NV^\varepsilon(\nabla u^n_{\delta}) \,,
\end{equation}
where  
$$
\varphi_n(x)=x\cdot\nabla u^n_{\delta}(x)-u^n_{\delta}(x) \,.
$$
Moreover, from the divergence theorem we have
\begin{equation}\label{angela}
\int_{\Omega_\delta} \nabla_\xi V^\varepsilon(\nabla u^n_{\delta})\cdot\nabla\varphi_n\, dx=-\int_{\Omega_\delta} \varphi_n\dive\left( \nabla_\xi V^\varepsilon(\nabla u^n_{\delta})\right)\, dx + \int_{\partial\Omega_\delta}\varphi_n\nabla_\xi V^\varepsilon(\nabla u^n_{\delta})\cdot\nu\, d\sigma \,.
\end{equation}
We are going to apply the divergence theorem in $\Omega_\delta$; to this end we set 
\begin{equation*}
\Gamma_{0,\delta}=\Gamma_0\cap\partial\Omega_\delta\, , \quad \Gamma_{1,\delta}=\Gamma_1\cap\partial\Omega_\delta \quad \text{and} \quad \Gamma_\delta=\partial\Omega_\delta\setminus( \Gamma_{0,\delta}\cup \Gamma_{1,\delta})\, .
\end{equation*}
From \eqref{angela} and by integrating \eqref{div_smart} in $\Omega_\delta$ we obtain
\begin{equation*}
\begin{aligned}
\int_{\Omega_\delta} \nabla_\xi V^\varepsilon(\nabla u^n_{\delta})\cdot\nabla\varphi_n\, dx=&-N\int_{\Omega_\delta} V^\varepsilon(\nabla u^n_{\delta})\, dx -\int_{\Omega_\delta}\dive\left(\varphi_n\nabla_\xi V^\varepsilon(\nabla u^n_{\delta})\right)\, dx \\
&+\int_{\Omega_\delta}\dive\left(xV^\varepsilon(\nabla u^n_{\delta})\right)\, dx \,, \\
\end{aligned}
\end{equation*}
and from $x\cdot\nu=0$ on $\Gamma_{1,\delta}$, we find
\begin{equation*}
\begin{aligned}
\int_{\Omega_\delta} \nabla_\xi V^\varepsilon(\nabla u^n_{\delta})\cdot\nabla\varphi_n\, dx
=&-N\int_{\Omega_\delta} V^\varepsilon(\nabla u^n_{\delta})\, dx-\int_{\Gamma_{0,\delta}\cup\Gamma_{1,\delta}}\varphi_n\nabla_\xi V^\varepsilon(\nabla u^n_{\delta})\cdot\nu\, d\sigma \\
&+\int_{\Gamma_{0,\delta}}V^\varepsilon(\nabla u^n_{\delta})x\cdot\nu\, d\sigma \\
&-\int_{\Gamma_\delta}[\varphi_n\nabla_\xi V^\varepsilon(\nabla u^n_{\delta})-xV^\varepsilon(\nabla u^n_{\delta})]\cdot\nu\, d\sigma \, .
\end{aligned}
\end{equation*}
By taking the limit as $\varepsilon\rightarrow 0$ and then as $n\rightarrow\infty$, using that $\nabla u\cdot\nu=0$ on $\Gamma_{1,\delta}$ (since $\partial_{\nu}u=0$ on $\Gamma_1$), we obtain
\begin{equation}\label{ryan}
\begin{aligned}
\int_{\Omega_\delta} \nabla_\xi V(\nabla u)\cdot\nabla\varphi\, dx=&-N\int_{\Omega_\delta} V(\nabla u)\, dx -\int_{\Gamma_{0,\delta}}\varphi\nabla_\xi V(\nabla u)\cdot\nu\, d\sigma +\int_{\Gamma_{0,\delta}}V(\nabla u)x\cdot\nu\, d\sigma \\
&-\int_{\Gamma_\delta}[\varphi\nabla_\xi V(\nabla u)-xV(\nabla u)]\cdot\nu\, d\sigma
\end{aligned}
\end{equation} 
where we let
\begin{equation} \label{phi_def_def}
\varphi(x)=x\cdot\nabla u(x)-u(x) \,.
\end{equation}
Now, we take the limit as $\delta\rightarrow 0$. Since $u\in W^{1,\infty}(\Omega)$ and $\mathcal{H}_{N-1}(\Gamma_\delta)$ goes to $0$ as $\delta \to 0$, we have that the last term in \eqref{ryan} vanishes and we obtain
\begin{equation*}
\int_{\Omega} \nabla_\xi V(\nabla u)\cdot\nabla\varphi\, dx=-N\int_{\Omega} V(\nabla u)\, dx -\int_{\Gamma_0}\varphi\nabla_\xi V(\nabla u)\cdot\nu\, d\sigma +\int_{\Gamma_{0}}V(\nabla u)x\cdot\nu\, d\sigma \,, \\
\end{equation*}
i.e. (in terms of $f$)
\begin{equation*}
\int_\Omega \frac{f'(|\nabla u|)}{|\nabla u|}\nabla u\cdot\nabla\varphi\, dx=-N\int_{\Omega}f(|\nabla u|)\, dx - \int_{\Gamma_0}\varphi\frac{f'(|\nabla u|)}{|\nabla u|}\partial_{\nu}u\, d\sigma +\int_{\Gamma_0}f(|\nabla u|)x\cdot\nu\, d\sigma.
\end{equation*}
Since $u$ satisfies \eqref{weak sol cono}, we get
\begin{equation}\label{stato_sociale}
\int_\Omega \varphi\, dx=-N\int_{\Omega}f(|\nabla u|)\, dx - \int_{\Gamma_0}\varphi\frac{f'(|\nabla u|)}{|\nabla u|}\partial_{\nu}u\, d\sigma +\int_{\Gamma_0}f(|\nabla u|)x\cdot\nu\, d\sigma.
\end{equation}
From \eqref{phi_def_def} and since $u=0$ on $\Gamma_0$ and $\partial_\nu u=0$ on $\Gamma_1$, we have
\begin{equation*}
\int_\Omega \varphi\, dx=-(N+1)\int_{\Omega}u\, dx
\end{equation*}
and 
\begin{equation} \label{annamo}
\int_{\Gamma_0}\varphi\frac{f'(|\nabla u|)}{|\nabla u|}\partial_{\nu}u\, d\sigma=\int_{\Gamma_0}f'(|\nabla u|)|\nabla u|x\cdot\nu\, d\sigma \,,
\end{equation}
where we used the expresion of the unit exterior normal on $\Gamma_0$ given by \eqref{normale}. From \eqref{annamo} and  \eqref{stato_sociale} we obtain
\begin{equation*}
-(N+1)\int_{\Omega}u\, dx+N\int_{\Omega}f(|\nabla u|)\, dx=-\int_{\Gamma_0}f'(|\nabla u|)|\nabla u|x\cdot\nu\, d\sigma+\int_{\Gamma_0}f(|\nabla u|)x\cdot\nu\, d\sigma.
\end{equation*}
which is \eqref{Pohozaev}, and the proof is complete.
\end{proof}

We conclude this subsection by exploiting the boundary condition $\partial_\nu u=0$ on $\Gamma_1$. Before doing this, we need to recall some notation from differential geometry (see also \cite[Appendix A]{ecker}). We denote by $D$ the standard Levi-Civita connection. Recall that, given an $(N-1)$-dimensional smooth orientable submanifold $M$ of $\mathbb{R}^N$ we define the \emph{tangential gradient} of a smooth function $f:M\rightarrow\mathbb{R}$ with respect to $M$ as
\begin{equation*}
\nabla^Tf(x)=\nabla f(x)-\nu\cdot\nabla f(x)\nu
\end{equation*}
for $x\in M$, where $\nabla f$ denotes the usual gradient of $f$ in $\mathbb{R}^N$ and $\nu$ is the outward unit normal at $x$ to $M$. Moreover, we recall that the \emph{second fundamental form} of $M$ is the bilinear and symmetric form defined on $TM\times TM$ as
\begin{equation*}
	\mathrm{II}(v,w)=D\nu(v) w \cdot \nu\, ;
\end{equation*} 
a submanifold is called \emph{convex} if the second fundamental form is non-negative definite.

\begin{lemma} \label{lemma_tangential}
Let $u$ be the solution to \eqref{pb cono}. Then 
\begin{equation}\label{piove}
\nabla_\xi V(\nabla u)\cdot\nu=0 \quad \text{on} \quad \Gamma_1 \,,
\end{equation}
and
\begin{equation}\label{grad_0}
\nabla (\nabla_\xi V(\nabla u)\cdot\nu)\cdot\nabla u=0 \quad \text{on} \quad \Gamma_1 \, .
\end{equation}
\end{lemma}

\begin{proof}
Since $\partial_\nu u=0$ on $\Gamma_1$, we immediately find \eqref{piove}.
By taking the tangential derivative in \eqref{piove} we get
\begin{equation*}
0=\nabla^T (\nabla_\xi V(\nabla u)\cdot\nu) 
=\nabla (\nabla_\xi V(\nabla u)\cdot\nu)-\nu\cdot \nabla(\nabla_\xi V(\nabla u)\cdot\nu)\nu \quad \text{on} \quad \Gamma_1\, .
\end{equation*}
By taking the scalar product with $\nabla u$ we obtain 
\begin{equation*}
0=\nabla (\nabla_\xi V(\nabla u)\cdot\nu)\cdot\nabla u-\nu\cdot \nabla(\nabla_\xi V(\nabla u)\cdot\nu)\partial_\nu u  
 \, ,
\end{equation*}
and since $u_\nu=0$ on $\Gamma_1$, we find \eqref{grad_0}.
\end{proof}

\subsection{Integral Identities for $S_2$} In this Subsection we prove some integral inequalities involving $S_2(W)$ and the solution to problem \eqref{pb cono}.

\begin{lemma}\label{formula_1}
Let $\Omega \subset \mathbb{R}^N$ be a sector-like domain and assume that $f$ satisfies \eqref{f_HP}. Let $u\in W^{1,\infty}(\Omega)$ be a solution of \eqref{pb cono} such that \eqref{key property} holds. Then the following inequality
\begin{equation} \label{integral_inequality_1}
2\int_{\Omega}S_2(W)u \, dx\geq -\int_{\Omega}S^2_{ij}(W)V_{\xi_i}(\nabla u)u_j\, dx
\end{equation}
holds. Moreover the equality sign holds in \eqref{integral_inequality_1} if and only if $\mathrm{II}(\nabla^T u,\nabla^T u)=0$ on $\Gamma_1$.
\end{lemma}

\begin{proof} We split the proof in two steps.

\emph{Step 1: the following identity} 
\begin{equation}\label{Bianchini_Ciraolo}
2\int_{\Omega}S_2(W)\phi \, dx=-\int_{\Omega}S^2_{ij}(W)V_{\xi_i}(\nabla u)\phi_j\, dx 
\end{equation}
\emph{holds for every} $\phi\in C^{1}_0(\Omega)$.

For $t>0$ we set $\Omega_t=\lbrace x\in\Omega \, : \, \dist(x,\partial\Omega)>t\rbrace$.
Let $\phi\in C^1_0(\Omega)$ be a test function and let $\varepsilon_0>0$ be such that  $\Omega_{\varepsilon_0}\subset\Omega$ and $supp(\phi)\subset\Omega_{\varepsilon_0}$. For $\varepsilon<\varepsilon_0$  sufficiently small, we set 
\begin{equation*}
a^i(x)=V_{\xi_i}(\nabla u(x)) \quad \text{ for every } \,  i=1,\dots,N, \, x\in\Omega.
\end{equation*}
From \eqref{key property} we have that $a^i \in W^{1,2}(\Omega)$, $i=1,\dots,N$. With this notation, the elements $w_{ij}=\partial_{j} V_{\xi_i}(\nabla u)$ of the matrix $W$ are given by 
$$
w_{ij}=\partial_j a^i \,.
$$ 
Let $\lbrace\rho_\varepsilon\rbrace$ be a family of mollifiers and define $a^i_\varepsilon=a^i\ast\rho_\varepsilon$. Let $W^\varepsilon=(w^{\varepsilon}_{ij})_{i,j=1,\dots,N}$ where $w^{\varepsilon}_{ij}=\partial_j a_\varepsilon^i$, and notice that
\begin{equation*}
	a^i_\varepsilon\rightarrow a^i \quad \text{ in } W^{1,2}(\Omega_{\varepsilon_0}) \quad \text{ and } \quad W^\varepsilon\rightarrow W \quad \text{ in } L^{2}(\Omega_{\varepsilon_0})\, ,
\end{equation*}
as $\varepsilon \to 0$. Moreover
\begin{equation}\label{ug_traccia}
\Tr W^\varepsilon=\Tr W=-1
\end{equation}
for every $x\in\Omega_\varepsilon$. 

Let $i,j=1,\dots,N$ be fixed. We have
\begin{equation*}
\begin{aligned}
w^{\varepsilon}_{ji}w^{\varepsilon}_{ij}&=\partial_j(a^i_\varepsilon\partial_ia^j_\varepsilon)-a^i_\varepsilon\partial_j\partial_i a^j_\varepsilon \\
&=\partial_j(a^i_\varepsilon\partial_i a^j_\varepsilon)- a^i_\varepsilon\partial_i\partial_j a^j_\varepsilon \\
&=\partial_j(a^i_\varepsilon\partial_i a^j_\varepsilon)- a^i_\varepsilon\partial_i w^{\varepsilon}_{jj} \, ,
\end{aligned}
\end{equation*}
for every $x\in\Omega_\varepsilon$, and by summing up over $j=1,\dots,N$, using \eqref{ug_traccia} (hence $\partial_i\sum_j w^\varepsilon_{jj}=0$), we obtain
\begin{equation*}
\begin{aligned}
\sum_{j} w^{\varepsilon}_{ji} w^{\varepsilon}_{ij}&=\sum_{j}\partial_j(a^i_\varepsilon\partial_i a^j_\varepsilon)\\
&=w^{\varepsilon}_{ii}\Tr W^\varepsilon-\sum_j \partial_j(S^2_{ij}(W^\varepsilon)a^i_\varepsilon), \quad x\in\Omega_\varepsilon.
\end{aligned}
\end{equation*}
By summing over $i=1,\dots,N$, from \eqref{defS_2} we have
\begin{equation}\label{gesso}
2S_2(W^\varepsilon)=\sum_{i,j}\partial_j(S^2_{ij}(W^\varepsilon) a^i_\varepsilon), \quad x\in\Omega_\varepsilon.
\end{equation}
Since
\begin{equation*}
\int_{\Omega_{\varepsilon_0}}\partial_j(S^2_{ij}(W^\varepsilon) a^i_\varepsilon)\phi\, dx+\int_{\Omega_{\varepsilon_0}}S^2_{ij}(W^\varepsilon)a^i_ \varepsilon\phi_j\, dx =\int_{\partial\Omega_{\varepsilon_0}}S^2_{ij}(W^\varepsilon)a^i_\varepsilon\nu_j \phi   \, d\sigma=0 \,,
\end{equation*}
from \eqref{gesso} and by letting  $\varepsilon$ to zero, we obtain \eqref{Bianchini_Ciraolo}.

\emph{Step 2.}
Let $\delta>0$ and consider a cut-off funtion $\eta^\delta\in C^\infty_c(\Omega)$ such that $\eta^\delta=1$ in $\Omega_\delta$ and $|\nabla\eta^\delta|\leq\frac{C}{\delta}$ in $\Omega\setminus\Omega_\delta$ for some constant $C$ not depending on $\delta$. By taking $\phi(x)=u(x)\eta^\delta(x)$ for $x\in\Omega$ in \eqref{Bianchini_Ciraolo} we obtain
\begin{equation}\label{anderson}
2\int_{\Omega}S_2(W)u\eta^\delta \, dx=-\int_{\Omega}S^2_{ij}(W)V_{\xi_i}(\nabla u)u_j\eta^\delta\, dx -\int_{\Omega}S^2_{ij}(W)V_{\xi_i}(\nabla u)u\eta^{\delta}_j\, dx \, .
\end{equation}
From \eqref{key property} we have that $W \in L^2 (\Omega)$ and the dominated convergence theorem yields
\begin{equation}\label{anderson_bis}
2\int_{\Omega}S_2(W)u\eta^\delta \, dx\rightarrow 2\int_{\Omega}S_2(W)u \, dx
\end{equation}
as $\delta \to 0$. Analogously,
\begin{equation} \label{anderson_chi}
\int_{\Omega}S^2_{ij}(W)V_{\xi_i}(\nabla u) u_j\eta^\delta \, dx \rightarrow \int_{\Omega}S^2_{ij}(W)V_{\xi_i}(\nabla u) u_j \, dx
\end{equation}
as $\delta \to 0$. 

Now, we consider the last term in \eqref{anderson}. We write $\Omega$ in the following way: 
\begin{equation} \label{divisione_omega}
\Omega=A_0^\delta \cup A_1^\delta\,,
\end{equation}
where
\begin{equation*}
A_0^\delta=\lbrace x\in\Omega \, : \, \dist(x,\Gamma_0)\leq\delta\rbrace \quad \text{ and } \quad A_1^\delta=\Omega \setminus A_0^\delta.
\end{equation*}
Since $u=0$ on $\Gamma_0$, we get that 
$$
u(x)\leq ||u||_{W^{1,\infty}(\Omega)}\, \dist(x,\Gamma_0)\leq ||u||_{W^{1,\infty}(\Omega)}\, \delta
$$ 
for every $x\in A_0^\delta$ and we obtain 
\begin{equation*} 
 \left|\int_{A_0^\delta}S^2_{ij}(W)V_{\xi_i}(\nabla u) u\eta^{\delta}_j\, dx\right|\leq C_1 |A_0^\delta| \,,
\end{equation*}
where $C_1$ is a constant depending on $||u||_{W^{1,\infty}(\Omega)}$ and $\|W\|_{L^2(\Omega)}$,
which implies that
\begin{equation}\label{anderson_tris}
\lim_{\delta \to 0} \int_{A_0^\delta}S^2_{ij}(W)V_{\xi_i}(\nabla u) u\eta^{\delta}_j\, dx  = 0\, .
\end{equation}

Now we show that 
\begin{equation}\label{adesso}
 \lim_{\delta \to 0} \int_{A_1^\delta}S^2_{ij}(W(x))V_{\xi_i}(\nabla u(x)) u(x)\eta^{\delta}_j(x)\, dx\geq 0\, .
\end{equation}
By choosing $\delta$ small enough, a point $x\in A_1^\delta$ can be written in the following way: $x=\bar{x}+t\nu(\bar{x})$ where $\bar{x}=\bar{x}(x)\in\Gamma_1$ and $t=|x-\bar x|$ with $0<t<\delta$. Moreover, by using a standard approximation argument, $\eta^\delta$ can be chosen in such a way that $\eta^\delta(x)=\frac{1}{\delta}\dist(x,\Gamma_1)$ for any $x\in A_1^\delta$, so that  
\begin{equation} \label{gradiente_eta_delta}
\nabla\eta^\delta(x)=-\frac{1}{\delta}\nu(\bar x) \,,
\end{equation} 
for every $x\in A_1^\delta \setminus \Omega_\delta$. For simplicity of notation we set $F=(F_1, \ldots, F_N)$, where
\begin{equation} \label{Fj}
F_j(x)=u(x) S^2_{ij}(W(x))V_{\xi_i}(\nabla u(x))
\end{equation}
for $j=1,\ldots,N$, and hence
\begin{equation}\label{presidenza}
\int_{A_1^\delta}S^2_{ij}(W)V_{\xi_i}(\nabla u) u\eta^{\delta}_j\, dx= \int_{A_1^\delta} F(x)\cdot\nabla\eta^\delta(x)\, dx\, .
\end{equation}
Since $\nabla \eta^\delta = 0 $ in $\Omega_\delta$ and $\nabla\eta^\delta(x)=-\frac{1}{\delta}\nu(\bar x)$, for every $x\in A_1^\delta \setminus \Omega_\delta$, we have
\begin{equation*}
\begin{aligned}
\int_{A_1^\delta}F(x)\cdot\nabla\eta^\delta(x)\, dx &=-\dfrac{1}{\delta}\int_{A_1^\delta \setminus \Omega_\delta} F(x)\cdot\nu(\bar x)\, dx \\
&=-\dfrac{1}{\delta}\int_{0}^{\delta}\, dt \int_{\lbrace x\in A_1^\delta \, : \, \dist(x,\Gamma_1)=t\rbrace}F(x)\cdot\nu(\bar{x})\, d\sigma \\
\end{aligned}
\end{equation*}
where we used coarea formula. Since we are in a \emph{small} $\delta$-tubular neighborhood of (part of) $\Gamma_1$, we can parametrize $A_1^\delta \setminus \Omega_\delta$ over (part of) $\Gamma_1$ as from \cite[Formula 14.98]{GT} we obtain that 
\begin{equation}\label{co-area}
\int_{A_1^\delta}F(x)\cdot\nabla\eta^\delta(x)\, dx =-\dfrac{1}{\delta}\int_{0}^{\delta}\, dt \int_{\Gamma_1}F(\bar{x}+t\nu(\bar{x}))\cdot\nu(\bar{x})|\det(Dg)|\, d\sigma \,.
\end{equation}

We notice that, by using this notation, proving \eqref{adesso} is equivalent to prove 
\begin{equation} \label{aimstep2}
 \lim_{\delta \to 0} \int_{A_1^\delta}F(x)\cdot\nabla\eta^\delta(x)\, dx \geq 0 \,,
\end{equation}
for $\delta>0$ sufficiently small. 

From \eqref{gradiente_eta_delta}, \eqref{Fj} and the definition of $S^2_{ij}$ \eqref{def_S^2}, we have
\begin{equation*}
\begin{aligned}
F(x) \cdot  \nu (\bar x) & =  -\delta_{ij}V_{\xi_i}(\nabla u(x)) u(x)\nu_j(\bar{x}) - w_{ji}(x) V_{\xi_i}(\nabla u(x)) u(x)\nu_j(\bar{x})  \\
& = -\left\lbrace\delta_{ij}V_{\xi_i}(\nabla u(x)) u(x)\nu_j(\bar{x}) + u(x)\frac{f'(|\nabla u(x)|)}{|\nabla u(x)|}w_{ji}(x) u_i(x)\nu_j(\bar{x})\right\rbrace 
\end{aligned}
\end{equation*}
for almost every $x=\bar{x}+t\nu(\bar{x})\in A_1^\delta \setminus \Omega_\delta$, with $0 \leq t \leq \delta$. Since
\begin{equation*}
w_{ij}\nu_i u_j=\partial_j(V_{\xi_i}(\nabla u)\nu_i)u_j-V_{\xi_i}(\nabla u)\partial_j \nu_i u_j\, ,
\end{equation*}
we have
\begin{equation} \label{air_cond}
\begin{aligned}
F(x) & \cdot  \nu (\bar x) = -u(x) \nabla_{\xi} V(\nabla u(x))\cdot\nu(\bar x)  \\
& -u(x)\frac{f'(|\nabla u(x)|)}{|\nabla u(x)|}\left\lbrace  \nabla (\nabla_\xi V(\nabla u(x))\cdot\nu(\bar x))\cdot\nabla u(x)-\frac{f'(|\nabla u(x)|)}{|\nabla u(x)|} \partial_j\nu_i(\bar x)u_j(x)u_i(x) \right\rbrace 
\end{aligned}
\end{equation}
for almost every $x=\bar{x}+t\nu(\bar{x})\in A_1^\delta \setminus \Omega_\delta$, with $0 \leq t \leq \delta$. Let 
$$ 
\Gamma_1^{\delta,t} = \lbrace x\in A_1^\delta \, : \, \dist(x,\Gamma_1)=t\rbrace \,.
$$ 
We notice that if $x\in \Gamma_1^{\delta,t}$ then $\nu(\bar x)=\nu^t(x)$ where $\nu^t(x)$ is the outward normal to $\Gamma_1^{\delta,t}$ at $x$. Hence 
\begin{equation}\label{moduli}
 \partial_j\nu_i(\bar x)u_j(x)u_i(x)=\mathrm{II}_x^{\delta,t}(\nabla^T u(x),\nabla^T u(x))
\end{equation}
where $\mathrm{II}_x^{\delta,t}$ is the second fundamental form of $\Gamma_{1}^{\delta,t}$ at $x$. Since $\Sigma$ is a convex cone then the second fundamental form of $\Gamma_1 \setminus \{O\}$ is non-negative definite. This implies that the second fundamental form of $\Gamma_1^{\delta,t}$ is non-negative definite for $t$ sufficiently small \cite[Appendix 14.6]{GT} and hence
\begin{equation}\label{sole}
 \partial_j\nu_i(\bar x)u_j(x)u_i(x)\geq 0  \,.
\end{equation}
From \eqref{sole} and \eqref{air_cond} we obtain 
\begin{equation}\label{lunga_lunga}
F(x)  \cdot  \nu (\bar x) \geq -u(x) \nabla_{\xi} V(\nabla u(x))\cdot\nu(\bar x)  -u(x)\frac{f'(|\nabla u(x)|)}{|\nabla u(x)|}  \nabla (\nabla_\xi V(\nabla u(x))\cdot\nu(\bar x))\cdot\nabla u(x) 
\end{equation}
for almost every $x=\bar{x}+t\nu(\bar{x})\in A_1^\delta \setminus \Omega_\delta$, with $0 \leq t \leq \delta$. We use \eqref{lunga_lunga} in the right-hand side of \eqref{co-area} and, by taking the limit as $\delta \to 0$, we obtain 
\begin{equation*}
 \lim_{\delta \to 0} \int_{A_1^\delta}F(x)\cdot\nabla\eta^\delta(x)\, dx  \geq -\int_{\Gamma_1}u \left( \nabla_{\xi} V(\nabla u)\cdot\nu + \frac{f'(|\nabla u|)}{|\nabla u|}\nabla (\nabla_\xi V(\nabla u)\cdot\nu)\cdot\nabla u \right) d\sigma\,.
\end{equation*}
From \eqref{piove} and \eqref{grad_0} we find \eqref{aimstep2}, and hence \eqref{adesso}. From \eqref{anderson}, \eqref{anderson_bis}, \eqref{anderson_chi}, \eqref{divisione_omega}, \eqref{anderson_tris} and \eqref{adesso}, we obtain \eqref{integral_inequality_1}.

\end{proof}

\section{Proof of Theorem \ref{teo 1 cono}} \label{section_proofthm1}

\begin{proof}[Proof of Theorem \ref{teo 1 cono}]
We divide the proof in two steps. We first show that 
\begin{equation}\label{Uguaglianza}
W=-\dfrac{1}{N}Id \quad \text{ a.e. in } \Omega.
\end{equation}
and 
\begin{equation} \label{step1_II}
\mathrm{II}(\nabla^T u, \nabla^T u) = 0 \quad \text{ on } \Gamma_1\,,
\end{equation} 
and then we exploit \eqref{Uguaglianza} in order to prove that $u$ is indeed radial. 

\emph{Step 1.} 
Let $g$ be the Fenchel conjugate of $f$ (in our case $g'=(f')^{-1}$), using \eqref{derivatediV}	we get that
	\begin{equation*}
	\begin{aligned}
		\dive\left(g(|\nabla_{\xi}V(\nabla u)|)\nabla_{\xi}V(\nabla u)\right)&= g'(|\nabla_{\xi}V(\nabla u)|)\nabla|\nabla_{\xi}V(\nabla u)|V_{\xi_j}(\nabla u) + g(|\nabla_{\xi}V(\nabla u)|)\Tr(W)\\
		&= g'(f'(|\nabla u|))\dfrac{V_{\xi_i}(\nabla u)}{|\nabla_{\xi}V(\nabla u)|}\partial_j(V_{\xi_i}(\nabla u))V_{\xi_j}(\nabla u) + g(f'(|\nabla u|))\Tr(W) \,,
	\end{aligned}
	\end{equation*}
a.e. in $\Omega$, where we used \eqref{derivatediV}. Since $\partial_j V_{\xi_i}(\nabla u) = w_{ij}$ and $g'=(f')^{-1}$, we obtain
$$
\dive\left(g(|\nabla_{\xi}V(\nabla u)|)\nabla_{\xi}V(\nabla u)\right) = u_iw_{ij}V_{\xi_j}(\nabla u)+g(f'(|\nabla u|))\Tr(W)
$$ 
a.e. in $\Omega$, and using again \eqref{derivatediV} we find
$$
\dive\left(g(|\nabla_{\xi}V(\nabla u)|)\nabla_{\xi}V(\nabla u)\right) =\dfrac{f'(|\nabla u|)}{|\nabla u|}u_i w_{ij}u_j+g(f'(|\nabla u|))\Tr(W)
$$	
a.e. in $\Omega$. Since 
\begin{equation} \label{gfprimo}
g(f'(t))=tf'(t)-f(t)
\end{equation} 
and $\Tr(W)=-1$, we obtain
\begin{equation} \label{nuova}
\dive\left(g(|\nabla_{\xi}V(\nabla u)|)\nabla_{\xi}V(\nabla u)\right)= \dfrac{f'(|\nabla u|)}{|\nabla u|}u_i w_{ij}u_j+f(|\nabla u|) - |\nabla u|f'(|\nabla u|)
\end{equation}
a.e. in $\Omega$. 

Since \eqref{defS_ij}, \eqref{derivatediV} and \eqref{Lfu_trW} yield 
\begin{equation*}
-S^2_{ij}(W)V_{\xi_i}(\nabla u)u_j=\dfrac{f'(|\nabla u|)}{|\nabla u|} w_{ji}u_iu_j+ f'(|\nabla u|)|\nabla u|\, ,
\end{equation*}
a.e. in $\Omega$, from \eqref{nuova} we obtain
\begin{equation} \label{nuovissima}
-S^2_{ij}(W)V_{\xi_i}(\nabla u)u_j= \dive\left(g(|\nabla_{\xi}V(\nabla u)|)\nabla_{\xi}V(\nabla u)\right) +2 f'(|\nabla u|)|\nabla u| -  f(|\nabla u|) \,, \end{equation}
a.e. in $\Omega$.

From Lemma \ref{formula_1} and \eqref{nuovissima}, we obtain
\begin{equation*}
\begin{aligned}
2\int_{\Omega}S_2(W)u\, dx\geq&- \int_{\Omega}S^2_{ij}(W)V_{\xi_i}(\nabla u)u_j\, dx \\
=&\int_{\partial\Omega}g(|\nabla_{\xi}V(\nabla u)|)\nabla_{\xi}V(\nabla u)\cdot\nu\, d\sigma+\int_{\Omega} \left[2f'(|\nabla u|)|\nabla u| - f(|\nabla u|) \right]\, dx  \,.
\end{aligned}
\end{equation*}
From \eqref{derivatediV} and \eqref{piove} we find 
$$
2\int_{\Omega}S_2(W)u\, dx\geq \int_{\Gamma_0}g(|\nabla_{\xi}V(\nabla u)|)\dfrac{f'(|\nabla u|)}{|\nabla u|}\partial_\nu u\, d\sigma+\int_{\Omega} \left[2f'(|\nabla u|)|\nabla u| - f(|\nabla u|) \right]\, dx \,.
$$
From \eqref{derivatediV} and \eqref{overdetermined cond} we have
$$
2\int_{\Omega}S_2(W)u\, dx\geq  -g(f'(c))f'(c)|\Gamma_0|+\int_{\Omega} \left[2f'(|\nabla u|)|\nabla u| - f(|\nabla u|) \right]\, dx 
$$
and from \eqref{gfprimo} we obtain
\begin{equation}\label{eq1}
2\int_{\Omega}S_2(W)u\, dx\geq-[cf'(c)-f(c)]f'(c)|\Gamma_0|+\int_{\Omega} \left[2f'(|\nabla u|)|\nabla u| - f(|\nabla u|) \right]\, dx \,.\end{equation}
From the Pohozaev identity \eqref{Pohozaev} and \eqref{gradiente cost} we get
\begin{equation*}
(N+1)\int_\Omega u\, dx-N\int_\Omega f(|\nabla u|)\, dx=(f'(c)c-f(c))N|\Omega|\, ;
\end{equation*}
which we use in \eqref{eq1} to obtain
\begin{equation}\label{eq1_bis}
2\int_{\Omega}S_2(W)u\, dx\geq-\dfrac{f'(c)|\Gamma_0|}{N|\Omega|}\int_\Omega \left[(N+1) u-Nf(|\nabla u|) \right] dx+\int_{\Omega} \left[2 f'(|\nabla u|)|\nabla u| - f(|\nabla u|) \right]\, dx \,.
\end{equation}
We notice that from \eqref{valore_c} we have
\begin{equation*}
|\Omega|=f'(c)|\Gamma_0|,
\end{equation*}
and from \eqref{eq1_bis} we obtain
\begin{equation}\label{eq1_1}
2\int_{\Omega}S_2(W)u\, dx\geq- \dfrac{N+1}{N} \int_\Omega u \, dx+2\int_{\Omega}f'(|\nabla u|)|\nabla u|\, dx\,.
\end{equation}
By using $u$ as a test function in \eqref{weak sol cono} we have that
$$
\int_\Omega u \, dx = \int_{\Omega}f'(|\nabla u|)|\nabla u| \, dx\,,
$$
and from \eqref{eq1_1} we find
\begin{equation}\label{Huisken}
2\int_{\Omega}S_2(W)u\, dx\geq\dfrac{N-1}{N}\int_\Omega u\, dx \,.
\end{equation}
From \eqref{matrices} and using the fact that $\Tr(W)=L_f u=-1$, we get that also the reverse inequality
\begin{equation}\label{legame tra u e S_2}
\dfrac{N-1}{N}\int_\Omega u\, dx\geq\int_\Omega 2S_2(W)u\, dx
\end{equation}
holds. From \eqref{Huisken} and \eqref{legame tra u e S_2}, we conclude that the equality sign must hold in \eqref{Huisken} and \eqref{legame tra u e S_2}. From Lemma \ref{lemma matrici generiche} we have that 
$$
W=\frac{\Tr(W)}{N} Id 
$$
a.e. in $\Omega$, and since $\Tr(W)=-1$ we obtain \eqref{Uguaglianza}. Moreover, Lemma \ref{formula_1} yields \eqref{step1_II}.

\medskip

\emph{Step 2: $u$ is a radial function.} 
From \eqref{Uguaglianza} we have that 
\begin{equation*}
-\dfrac{1}{N}\delta_{ij}=\partial_{j} V_{\xi_i}(\nabla u(x)) \, ,
\end{equation*}
for every $i, j=1,\dots, N$, which implies that there exists $x_0\in\mathbb{R}^N$ such that
\begin{equation*}
\nabla_{\xi} V(\nabla u(x))=-\dfrac{1}{N}(x-x_0),
\end{equation*}
i.e. according to \eqref{derivatediV}
\begin{equation*}
\dfrac{f'(|\nabla u(x)|)}{|\nabla u(x)|}\nabla u(x)=-\dfrac{1}{N}(x-x_0) \,.
\end{equation*}
Hence 
\begin{equation*}
\nabla u(x)=- g'\left(\dfrac{|x-x_0|}{N}\right) \frac{x-x_0}{|x-x_0|} \quad \text{ in } \Omega \,.
\end{equation*}
Since $u=0$ on $\Gamma_0$, we obtain \eqref{u_radial_Lf} and in particular $u$ is radial with respect to $x_0$. Moreover, from \eqref{step1_II} we find that $x_0$ must be the origin or, if $\partial \Sigma$ contains flat regions, a point on $\partial \Sigma$.
\end{proof}

\section{Cones in space forms: proof of Theorem \ref{teo 2 cono}} \label{section_spaceforms}
The goal of this section is to give an easily readable proof of Theorem \ref{teo 2 cono}. More precisely we assume more regularity on the solution than the one actually assumed in Theorem \ref{teo 2 cono} in order to give a coincise and clear idea of the proof in this setting, and we omit the technical details which are, in fact, needed. A rigorous treatment of the argument described below can be done by adapting the (technical) details in Section \ref{section_proofthm1} and in \cite{Pacella-Tralli}.

Before starting the proof we declare some notations we use in the statement of Theorem \ref{teo 2 cono} and we are going to adopt in the following. Given a $N$-dimensional Riemannian manifold $(M,g)$, we denote by $D$ the Levi-Civita connection of $g$. Moreover given a $C^2$-map $u:M\rightarrow\mathbb{R}$, we denote by $\nabla u$ the gradient of $u$, i.e. the dual field of the differential of $u$ with respect to $g$, and by $\nabla^2u=Ddu$ the Hessian of $u$. We denote by $\Delta$ the Laplace-Betrami operator induced by $g$; $\Delta u$ can be defined as the trace of $\nabla^2u$ with respect to $g$. 
Given a vector field $X$ on an oriented Riemannian manifold $(M, g)$, we denote by $\dive X$ the divergence of $X$ with respect to $g$. If $\lbrace e_k\rbrace$ is a local orthonormal frame on $(M, g)$, then
\begin{equation*}
\dive X=\sum_{k=1}^N g(D_{e_k}X,e_k)\, ;
\end{equation*}
notice that, if $u$ is a $C^1$-map and if $X$ is a $C^1$ vector field on $M$, we have the following \emph{integration by parts} formula
\begin{equation*}
\int_\Omega g(\nabla u,\nu)\, dx=-\int_\Omega u\dive X\, dx + \int_{\partial\Omega}ug(X,\nu)\, d\sigma\, ,
\end{equation*}
where $\nu$ is the outward normal to $\partial\Omega$ and $\Omega$ is a bounded domain which is regular enough. Here and in the following, $dx$ and $d\sigma$ denote the volume form of $g$ and the induced $(N-1)$-dimensional Hausdorff measure, respectively.

\begin{proof}[Proof of Theorem \ref{teo 2 cono}]
We divide the proof in four steps.

\medskip

\emph{Step 1: the $P$-function.}
Let $u$ be the solution to problem \eqref{pb cono sf} and, as in \cite{CV1}, we consider the $P$-function defined by 
	\begin{equation*} 
		P(u)=|\nabla u|^2+\dfrac{2}{N}u+Ku^2 \,.
	\end{equation*}
Following \cite[Lemma 2.1]{CV1}, $P(u)$ is a subharmonic function and, since $u=0$ on $\Gamma_0$ and from \eqref{gradiente cost}, we have that $P(u)=c^2$ on $\Gamma_0$. Moreover, 
\begin{equation}\label{posto29}
\nabla P(u)=2\Hess u\nabla u+\dfrac{2}{n}\nabla u+2Ku\nabla u \,.
\end{equation} 
From the convexity assumption of the cone $\Sigma$, we have that 
\begin{equation}\label{posto29E}
g(\Hess u \nabla u,\nu) \leq 0 \,.
\end{equation}
Indeed, since $u_\nu = 0$ on $\Gamma_1$ and by arguing as done for \eqref{grad_0}, we obtain that 
\begin{equation*}
0= g( \nabla u_\nu , \nabla u)= g(\Hess u \nabla u,\nu) + \mathrm{II}(\nabla u , \nabla u) \geq  g(\Hess u \nabla u,\nu)   \quad \text{on} \quad \Gamma_1 \, ,
\end{equation*}
which is \eqref{posto29E}. From \eqref{posto29} and \eqref{posto29E} we obtain 
\begin{equation*}
\partial_{\nu}P(u)=2 g(\Hess u \nabla u,\nu) + \dfrac{2}{n}\partial_\nu u+2Ku\partial_\nu u\leq 0 \quad \text{ in } \, \Gamma_1\setminus\lbrace O\rbrace  \,.
\end{equation*}
Hence, the function $P$ satisfies:
\begin{equation*}
\begin{cases}
\Delta P(u)\geq 0 &\mbox{in } \Omega, \\ P(u)=c^2 &\mbox{on } \Gamma_0 \\ \partial_{\nu}P(u)\leq 0 \, &\mbox{on } \Gamma_1\setminus\lbrace O\rbrace \,.
\end{cases}
\end{equation*}
Moreover, again from  \cite[Lemma 2.1]{CV1}, we have that 
\begin{equation} \label{uguaglianza}
\Delta P(u) = 0 \quad \text{ if and only if } \quad \Hess u = \left(-\frac{1}{N}-Ku\right) g \,.
\end{equation}

\medskip

\emph{Step 2: we have }
\begin{equation}
P(u)\leq c^2 \quad \textit{ in } \, \Omega.
\end{equation}
Indeed, we multiply $\Delta P(u) \geq 0$ by $(P(u)-c^2)^+$ and by integrating by parts we obtain
$$
0 \geq \int_{\Omega \cap \{P>c^2\}} |\nabla P|^2\, dx - \int_{\partial \Omega}  (P(u)-c^2)^+ \partial_\nu P\, d\sigma \,.
$$
Since $P(u)=c^2$ on $\Gamma_0$ and  $\partial_{\nu}P(u)\leq 0 $ on $\Gamma_1$ we obtain that
$$
0 \geq \int_{\Omega \cap \{P>c^2\}} |\nabla P|^2\, dx \geq 0
$$
and hence $P(u) \leq c^2$.

\medskip

\emph{Step 3: $P(u)=c^2$}.
By contradiction, we assume that $P(u)<c^2$ in $\Omega$. Since $\dot h>0$, we have
\begin{equation*}
c^2 \int_\Omega \dot h\, dx >  \int_\Omega \dot h |\nabla u|^2\, dx + \frac{2}{n} \int_\Omega \dot h u\, dx + K \int_\Omega \dot h u^2\, dx \, . 
\end{equation*}
Since
\begin{equation*}
\dive(\dot h u \nabla u)= \dot h |\nabla u|^2 + \dot h u \Delta u + \ddot h u \partial_r u  
\end{equation*}
and
$$
\ddot h = -K h \,,
$$ 
and from $u=0$ on $\Gamma_0$ and $\partial_{\nu}u=0$ on $\Gamma_1\setminus\{ O\}$, we have that 
\begin{equation*}
\begin{split}
c^2 \int_\Omega \dot h\, dx  & >  - \int_\Omega \dot h u \Delta u\, dx - \int_\Omega \ddot h u \partial_r u\, dx + \frac{2}{n} \int_\Omega \dot h u\, dx + K \int_\Omega \dot h u^2\, dx  \\  & = (n+1)K \int_\Omega \dot h u^2\, dx +  \left(1 + \frac 2n \right) \int_\Omega \dot h u\, dx  + K \int_\Omega  h u \partial_r u\, dx \,.
\end{split}
\end{equation*}
From $\dive (h\partial_r) = n \dot h$ we have
\begin{equation*}
\dive (u^2 h \partial_r) = n \dot h u^2 + 2 h u \partial_r u \,,
\end{equation*} 
and from $u=0$ on $\Gamma_0$ and $\partial_{\nu}u=0$ on $\Gamma_1\setminus\{ O\}$ we obtain 
\begin{equation}\label{minore_stretto}
c^2 \int_\Omega \dot h\, dx  >  \left(1 + \frac 2n \right)\left(  \int_\Omega \dot h u\, dx  - K \int_\Omega  h u \partial_r u \, dx\right) \,.
\end{equation}
Now we show that if $u$ is a solution of \eqref{pb cono sf} satisfying \eqref{overdetermined cond spaceforms} then the equality sign holds in \eqref{minore_stretto}. Indeed, let $X=h\partial_r$ be the radial vector field and, by integrating formula (2.8) in \cite{CV1}, we get
\begin{equation*} 
- \frac{c^2}{n} \int_{\partial \Omega} g(X,\nu)\, d\sigma
+ \frac{n+2}{n} \int_\Omega \dot h u\, dx 
-( n-2)  K  \int_\Omega \dot h u^2\, dx + \left(\frac{2}{n}-3\right) K \int_\Omega u g(X,\nabla u)\, dx  = 0\,.
\end{equation*}
Since $\dive X=n \dot h$ we obtain
\begin{equation*} 
c^2 \int_\Omega \dot h\, dx =  \frac{n+2}{n} \int_\Omega \dot h u \, dx
-( n-2)  K  \int_\Omega \dot h u^2\, dx + \left(\frac{2}{n}-3\right) K \int_\Omega u g(X,\nabla u)  \, dx \,,
\end{equation*}
i.e.
\begin{equation*}
c^2 \int_\Omega \dot h \, dx =  \left(1 + \frac 2n \right)\left(  \int_\Omega \dot h u\, dx  - K \int_\Omega  h u \partial_r u\, dx \right) \,,
\end{equation*}
where we used that $u=0$ on $\Gamma_0$, $\partial_{\nu}u=0$ on $\Gamma_1\setminus\{ O\}$ and $g(X,\nu)=0$ on $\Gamma_1$.
From \eqref{minore_stretto} we find a contradiction and hence $P(u)\equiv c^2$ in $\Omega$.

\medskip

\emph{Step 4: $u$ is radial}. Since $P(u)$ is constant, then $\Delta P(u) = 0$ and from \eqref{uguaglianza} we find that $u$ satisfies the following Obata-type problem
\begin{equation} \label{luigifava}
\begin{cases}
\Hess u = (-\frac{1}{N}-Ku) g & \text{in } \Omega \,,\\
u=0 & \text{on } \Gamma_0 \,, \\
\partial_{\nu}u=0 & \text{on } \Gamma_1\setminus\{O\} \,.
\end{cases}
\end{equation}
We notice that the maximum and the minimum of $u$ can not be both achieved on $\Gamma_0$ since otherwise we would have that $u\equiv 0$. Hence, at least one between the maximum and the minimum of $u$ is achieved at a point $p\in\Omega\cup\Gamma_1$. Let $\gamma:I\rightarrow M$ be a unit speed maximal geodesic satisfying $\gamma(0)=p$ and let $f(s)=u(\gamma(s))$. From the first equation of \eqref{luigifava} it follows 
\begin{equation*}
f''(s)=-\dfrac{1}{N}-Kf(s) \,.
\end{equation*}
Moreover, the definition of $f$ and the fact that $\nabla u(p)=0$ yield
\begin{equation*}
f'(0)=0 \quad \text{ and } \quad f(0)=u(p),
\end{equation*}
and therefore 
\begin{equation*}
f(s)=\left(u(p)-\dfrac{1}{N}\right)H(s)-\dfrac{1}{N}.
\end{equation*}
This implies that $u$ has the same expression along any geodesic strating from $p$, and hence $u$ depends only on the distance from $p$. This means that $\Omega=\Sigma\cap B_{R}$ where $B_R$ is a geodesic ball and $u$ depends only on the distance from the center of $B_R$.
\end{proof}

\bigskip 

\noindent{\textbf{Acknowledgement.}} The authors wish to thank Luigi Vezzoni for suggesting useful remarks regarding Section \ref{section_spaceforms}. The authors have been partially supported by the ``Gruppo Nazionale per l'Analisi Matematica, la Probabilit\`a e le loro Applicazioni'' (GNAMPA) of the ``Istituto Nazionale di Alta Matematica'' (INdAM, Italy).

\end{document}